\documentclass[12pt]{article}  
\usepackage{amsmath}
\usepackage{amssymb}
\usepackage{theorem}
\usepackage{euscript}
\usepackage{pstricks}
\usepackage{authblk}
\usepackage{verbatim}
\usepackage{graphics}
\usepackage{graphicx}
\usepackage{color}

\usepackage{hyperref}
\hypersetup
{
	colorlinks,
	citecolor=black,
	filecolor=black,
	linkcolor=blue,
	urlcolor=blue
}
\hypersetup{linktocpage}
\newtheorem{theorem}{Theorem}[section]
\newtheorem{lemma}[theorem]{Lemma}

\numberwithin{equation}{section}

\theoremstyle{definition}
 
\theoremstyle{remark}
\newtheorem{remark}[theorem]{Remark}

\newcommand{\brac}[1]{\left(#1\right)}

\newcommand{\brab}[1]{\left\{#1\right\}}

\newcommand{\bg}{{\boldsymbol{g}}}

\newcommand{\bk}{{\boldsymbol{k}}}

\newcommand{\bs}{{\boldsymbol{s}}}

\newcommand{\bx}{{\boldsymbol{x}}}

\newcommand{\bF}{{\boldsymbol{F}}}
\newcommand{\bX}{{\boldsymbol{X}}}
\newcommand{\bw}{{\boldsymbol{w}}}
\newcommand{\by}{{\boldsymbol{y}}}
\newcommand{\bv}{{\boldsymbol{v}}}

\newcommand{\bW}{{\boldsymbol{W}}}

\newcommand{\bone}{{\boldsymbol{1}}}

\newcommand{\bxi}{{\boldsymbol{\xi}}}
\newcommand{\rd}{{\rm d}} 
\newcommand{\Int}{{\rm Int}}

\def\ZZd{{\mathbb Z}^d}

\def\RR{{\mathbb R}}
\def\RRd{{\mathbb R}^d}
\def\RRdp{{\mathbb R}^d_+}
\def\NN{{\mathbb N}}
\def\NNd{{\NN}^d}

\def\NN{{\mathbb N}}
\def\RR{{\mathbb R}}

\def\UU{{\mathbb U}}

\def\NNd{{\mathbb N}^d}
\def\RRd{{\mathbb R}^d}

\def\UUd{{\mathbb U}^d}

\def\ZZd{{\mathbb Z}^d}

\def\RRdp{{\mathbb R}^\infty_+}

\def\Qq{{\mathcal Q}}

\def\NN{{\mathbb N}}
\def\RR{{\mathbb R}}

\def\NNd{{\mathbb N}^d}
\def\RRd{{\mathbb R}^d}

\def\RRp{{\mathbb R}_+}
\def\RRdp{{\mathbb R}^d_+}

\newcommand{\norm}[2]{\left\|{#1}\right\|_{#2}}

\title{\sffamily Laguerre- and Laplace-weighted integration of mixed-smoothness functions }

\author[a]{Dinh D\~ung}
\affil[a]{Information Technology Institute, Vietnam National University, Hanoi
	\protect\\
	144 Xuan Thuy, Cau Giay, Hanoi, Vietnam
	\protect\\
	Email: dinhzung@gmail.com}

\date{\today}
\begin{document}
\maketitle

\begin{abstract}
 We investigate the  approximation of generalized Laguerre- or Laplace-weighted integrals over  $\mathbb{R}^d_+$ or  $\mathbb{R}^d$  of functions from generalized Laguerre- or Laplace-weighted  Sobolev spaces of mixed smoothness, respectively. We prove upper and lower bounds  of the convergence rate of optimal quadratures with respect to $n$ integration nodes for functions from  these spaces.  The upper bound is performed by sparse-grid quadratures with integration nodes on step hyperbolic corners or  hyperbolic crosses in the function domain $\mathbb{R}^d_+$ or  $\mathbb{R}^d$, respectively.
	
	
	
\end{abstract}

\section{Introduction}
\label{Introduction}

We first introduce generalized Laguerre- and Laplace-weighted  Sobolev spaces of  mixed smoothness.  Let $\UUd$ be either $\RRd$ or $\RRdp$.
Let 
\begin{equation} \label{bw(bx)}
	\bw_r(\bx):= \prod_{i=1}^d w_r(x_i), \ \ \bx \in \UUd,
\end{equation}
where
\begin{equation} \label{w(x)}
	w_r(x):= 	|x|^{\alpha + r/2}\exp (- a|x|+ b), \ \ \alpha > -1, \ \ r \in \NN_0, \ \ a > 0, \ \ b \in \UU,
\end{equation} 
is a univariate generalized Laguerre-Laplace-weight. 	The formula \eqref{w(x)} combines two different weights defined on different domains: the generalized Laguerre-weight $x^{\alpha + r/2}\exp (- a x + b)$ on $\RRp$, and the generalized Laplace-weight $|x|^{\alpha + r/2}\exp (- a|x|+ b)$ on $\RR$.
However, they are connected by approximation and integration properties as in particular, shown in the present paper.
In what follows, for simplicity of presentation, 
without  loss of generality we  assume $a=1$ and $b=0$, and
fix the weight $w$ and hence the  parameters $\alpha$. The parameter $r$ is associated with the definition of weighted Sobolev space  $W^r_{p,\bw_r}(\Omega)$ of mixed smoothness $r$ which is introduced below. We also omit $r$ in the notation if $r=0$, i.e., $w:=w_0$ and $\bw:=\bw_0$.

Let  $1\leq p<\infty$ and $\Omega$ be a Lebesgue measurable set on $\UUd$. 
We denote by  $L_{p,\bw_r}(\Omega)$ the weighted space  of all functions $f$ on $\Omega$ such that the norm
\begin{align} \label{L-Omega}
\|f\|_{L_{p,\bw_r}(\Omega)} : = 
\bigg( \int_\Omega |f(\bx)\bw_r(\bx)|^p  \rd \bx\bigg)^{1/p} 
\end{align}
is finite. For $r \in \NN$, we define the weighted  Sobolev space $W^r_{p,\bw_r}(\Omega)$ of mixed smoothness $r$  as the normed space of all functions $f\in L_{p,\bw_r}(\Omega)$ such that the weak (generalized) partial derivative $D^{\bk} f$ belongs to $L_{p,\bw_r}(\Omega)$ for  every $\bk \in \NNd_0$ satisfying the inequality $|\bk|_\infty \le r$. The norm of a  function $f$ in this space is defined by
\begin{align} \label{W-Omega}
	\|f\|_{W^r_{p,\bw_r}(\Omega)}
	: =  \
	\Bigg(\sum_{|\bk|_\infty \le r} \|D^{\bk} f\|_{L_{p,\bw_r}(\Omega)}\Bigg)^{1/p}.
\end{align}
It is useful to notice that any function $f \in W^r_{p,\bw_r}(\UUd)$ is equivalent  in the sense of the Lesbegue measure to a continuous (not necessarily bounded) function on $\UUd$. Hence throughout the present paper,  we always assume that the functions $f \in W^r_{p,\bw_r}(\UUd)$ are  continuous. We need this assumption for  well-defined  quadratures for functions $f \in W^r_{p,\bw_r}(\UUd)$. 


In the present paper, we are interested in approximation of   weighted integrals 
\begin{equation} \label{If}
\int_{\UUd} f(\bx) \bw(\bx) \, \rd\bx 
\end{equation}
for functions $f$ lying  in the space 
$W^r_{1,\bw_r}(\UUd)$.
To approximate them we use quadratures  of the form
\begin{equation} \label{Q_nf-introduction}
	Q_kf: = \sum_{i=1}^k \lambda_i f(\bx_i), 
\end{equation}
where $\bx_1,\ldots,\bx_k \in \UUd$  are  the integration nodes and $\lambda_1,\ldots,\lambda_k$ the integration weights.  For convenience, we assume that some of the integration nodes may coincide. 

Let $\bF$ be a set of continuous functions on $\UUd$.  Denote by $\Qq_n$  the family of all quadratures $Q_k$ of the form \eqref{Q_nf-introduction} with $k \le n$. The optimality  of  quadratures from $\Qq_n$ for  $f \in \bF$  is measured by 
\begin{equation} \label{Int_n}
\Int_n(\bF) :=\inf_{Q_n \in \Qq_n} \ \sup_{f\in \bF} 
\bigg|\int_{\UUd} f(\bx) \bw(\bx) \, \rd\bx - Q_nf\bigg|.
\end{equation}

Throughout this paper, we use $\bW^r_{p,\bw_r}(\Omega)$  to denote the unit ball in the space $W^r_{p,\bw_r}(\Omega)$.

The aim of  the present paper is to construct  sparse-grid quadratures  for the approximation of   
Laguerre- or Laplace-weighted integrals over $\UUd$ of 
functions belonging to the weighted Sobolev spaces $W^r_{1,\bw_r}(\UUd)$ of mixed smoothness $r$.
We construct sparse-grid quadratures which give upper bounds for convergence rate of  best quadrature 
quantified by $\Int_n(\bW^r_{1,\bw_r}(\UUd))$. 

The problem of numerical weighted integration considered in the present paper is  related to  the research direction of optimal approximation and integration for functions having mixed smoothness on one hand, and the other research  direction of univariate weighted  approximation and integration on $\RR$, on the other hand.  For survey and bibliography, we refer the reader to the books  \cite{DTU18B,Tem18B} on the first direction, and \cite{Mha1996B,Lu07B,JMN2021} on the second one. The present paper can be considered as a continuation of the papers \cite{DD2023,DK2022,DD2025a} on numerical weighted integration of mixed-smoothness functions.
In particular,  in \cite{DD2023}, we have obtained  upper and lower bounds for the quantity $\Int_n(\bW^r_{1,\bv}(\RRd))$  of best quadrature of the integral 

The numerical problem of weighted integration addressed in this study connects two research directions. On one hand, it relates to optimal approximation and integration for functions with mixed smoothness. On the other hand, it aligns with the theory of univariate weighted approximation and integration on $\RR$.
 For a survey and bibliography, the reader is referred to the monographs \cite{DTU18B, Tem18B} for the former direction and to \cite{Mha1996B, Lu07B, JMN2021} for the latter. This work can be viewed as a continuation of the investigations reported in \cite{DD2023, DK2022, DD2025a} on numerical weighted integration of functions with mixed smoothness.

In particular, in \cite{DD2023} we established both upper and lower bounds for the quantity
$\Int_n(\bW^r_{1,\bv}(\RRd))$ representing the best 
 quadrature error for the integral
\begin{equation} \label{If-G}
	\int_{\RRd} f(\bx) \bv(\bx) \, \rd\bx 
\end{equation}
on the unit ball $\bW^r_{1,\bv}(\RRd)$ of the  Freud-type weighted Sobolev space $W^r_{1,\bv}(\RRd)$ of mixed smoothness $r$,
where
\begin{equation} \label{bv(bx)}
	\bv(\bx) := \prod_{i=1}^d v(x_i),
\ \ \
	v(x):= 	\exp (- a|x|^\lambda + b), \ \ \lambda > 1, \ \ a > 0, \ \ b \in \RR,
\end{equation}
and the space $W^r_{1,\bv}(\RRd)$ is defined in the manner of  \eqref{L-Omega}--\eqref{W-Omega} by replacing $\bw_r$ with $\bv$. However, the methodology employed in \cite{DD2023} to derive the upper bound is not suitable for the limited case $\lambda =1$ in which  the Freud-type weight $\bv$ reduces to the Laplace-weight.  In the present work, we address this gap and provide the corresponding bound for this special case.

The main results of the present paper are read as follows.
For  the set $\bW^r_{1,\bw_r}(\UUd)$, we prove the upper and  lower bounds for  convergence rate of optimal integration as
\begin{equation}\label{Int_n(W^r_1)}
n^{-3r/4} (\log n)^{3r(d-1)/4}
	\ll	
	\Int_n(\bW^r_{1,\bw_r}(\UUd)) 
	\ll 
	n^{-r/2} (\log n)^{(r/2 + 1)(d-1)}.
\end{equation}
The Smolyak  quadratures performing this upper bound  are constructed  on  sparse grids of  integration nodes  which form \emph{a step hyperbolic corner} for the generalized Laguerre-weighted integration or \emph{a step hyperbolic cross} for the generalized Laplace-weighted integration in the function domain $\UUd$. This development of the construction of sparse grids and the associated quadratures  was proposed first  by the author \cite{DD2023} 
for numerical Freud-type weighted integration of functions having weighted mixed smoothness.  It is remarkable to notice that these sparse grids are completely different from  classical hyperbolic crosses in  frequency domains and Smolyak sparse grids   in compact function domains. To prove the lower bound in  \eqref{Int_n(W^r_1)} we adopt a traditional technique to construct for arbitrary $n$ integration nodes a fooling function vanishing at these nodes.

The Smolyak quadratures that yield the upper bound are built on sparse grids of integration nodes that form a step hyperbolic corner for generalized Laguerre-weighted integration, or a step hyperbolic cross for generalized Laplace-weighted integration, within the function domain in the function domain $\UUd$.
 The construction of these sparse grids and the associated quadratures was first proposed by the author in \cite{DD2023} for numerical Freud-type weighted integration of functions with weighted mixed smoothness. Notably, these sparse grids differ radically from classical hyperbolic crosses in the frequency domain and from Smolyak sparse grids in compact domains. To prove the lower bound in  \eqref{Int_n(W^r_1)} we employ a traditional technique to construct for arbitrary $n$ integration nodes a fooling function vanishing at all of these nodes.

It is  informative to compare the results  \eqref{Int_n(W^r_1)} in the case when $\UUd=\RRd$ and $\alpha = 0$ in  the weight $\bw_r$, with the results obtained in \cite{DD2023}, on numerical Gaussian-weighted  integration of the integral
\eqref{If-G}
of functions $f$ belonging to Gaussian-weighted Sobolev space $W^r_{1,\bg}(\RRd)$ of mixed smoothness $r$,
where 
$\bg(\bx) = (2\pi)^{-d/2}\exp (- |\bx|_2^2/2)$
is the  density of the $d$-dimensional standard Gaussian measure.
For the convergence rate of $\Int_n(\bW^r_{1,\bg}(\RRd))$  of best quadrature of this integral on the unit ball $\bW^r_{1,\bg}(\RRd)$, it has been proven in \cite{DD2023} that there hold the upper and lower bounds
\begin{equation}\label{Int_n(W^r_1)-tau=2}
	n^{-r/2} (\log n)^{r(d-1)/2} 
	\ll	
	\Int_n(\bW^r_{1,\bg}(\RRd))
	\ll 
	n^{-r/2} (\log n)^{(r/2 + 1)(d-1)}.
\end{equation}
At first glance, the upper bound in \eqref{Int_n(W^r_1)-tau=2} bears similarity to that in  \eqref{Int_n(W^r_1)}. However, a crucial difference lies in the definitions of the spaces  $W^r_{1,\bw_r}(\RRd)$ and $W^r_{1,\bg}(\RRd)$. 
The weight $\bw_r$
is obtained from the weight $\bw$ by multiplying by  the factor $\prod_{i=1}^d |x_i|^{r/2}$ which explicitly depends on the mixed smoothness~$r$.
 This structural distinction explains why the  upper bounds for $\Int_n(\bW^r_{1,\bw_r}(\RRd))$ and $\Int_n(\bW^r_{1,\bg}(\RRd))$ may coincide. The lower bound and upper bound for convergence rate diverge in both polynomial and logarithm terms in \eqref{Int_n(W^r_1)}, while they  coincide in the main polynomial term in \eqref{Int_n(W^r_1)-tau=2}. Notably, in both frameworks the exact convergence rate of the optimal weighted integration remains an open problem.

Let us briefly describe the remaining sections.  In Section \ref{Weighted integration in one dimension}, we prove upper and lower bounds
 for $\Int_n(\bW^r_{1,w_r}(\UU))$ and construct  quadratures which perform the upper bound. In Section \ref{Weighted integration in high dimension}, we prove upper and lower bounds of $\Int_n(\bW^r_{1,\bw_r}(\UUd))$ for $d \ge 2$, and construct quadratures which give the upper bound.

\medskip
\noindent
{\bf Notation.} 
 Denote: $\bone:=(1,1,...,1) \in \RRd$, for $\bx \in \RRd$, $\bx=:\brac{x_1,...,x_d}$,
$|\bx|_\infty:= \max_{1\le j \le d} |x_j|$ and $|\bx|_2:= \brac{\sum_{j=1}^d |x_j|^2}^{1/2}$.  For $\bx, \by \in \RRd$, the inequality $\bx \le \by$ means $x_i \le y_i$ for every $i=1,...,d$.  
We use letters $C$  and $K$ to denote general positive constants which may take different values. For the quantities $A_n(f,\bk)$ and $B_n(f,\bk)$ depending on 
$n \in \NN$, $f \in W$, $\bk \in \ZZd$,  
we write  $A_n(f,\bk) \ll B_n(f,\bk)$, $f \in W$, $\bk \in \ZZd$ ($n \in \NN$ is specially dropped),  
if there exists some constant $C >0$ such that 
$A_n(f,\bk) \le CB_n(f,\bk)$ for all $n \in \NN$,  $f \in W$, $\bk \in \ZZd$ (the notation $A_n(f,\bk) \gg B_n(f,\bk)$ has the obvious opposite meaning), and  
$A_n(f,\bk) \asymp B_n(f,\bk)$ if $A_n(f,\bk) \ll B_n(f,\bk)$
and $B_n(f,\bk) \ll A_n(f,\bk)$.  Denote by $|G|$ the cardinality of the set $G$. 
For a Banach space $X$, denote by the boldface $\bX$ the unit ball in $X$.

	\section{Weighted integration in one dimension}
\label{Weighted integration in one dimension}

In this section, for one-dimensional numerical integration, we prove upper and lower bounds for $\Int_n\big(\bW^r_{1,w_r}(\UU)\big)$ and present some asymptotically optimal quadratures.

We start this section with a well-known inequality which is implied directly from the definition \eqref{Int_n} and which is quite useful  for lower estimation of $\Int_n(\bF)$.
 	Let $\bF$ be a set of continuous functions on $\UUd$. 	Then we have
 	\begin{equation} \label{Int_n>}
 		\Int_n(\bF) \ge \inf_{\brab{\bx_1,...,\bx_n} \subset \UUd} \ \sup_{f\in \bF: \ f(\bx_i)= 0,\ i =1,...,n}\bigg|\int_{\UUd} f(\bx) \bw(\bx) \, \rd\bx\bigg|.
 	\end{equation}
 
%
We first consider the problem of approximation of  integral \eqref{If} for univariate functions from 
$W^r_{1,w_r}(\RRp)$. Let $(p_m(w))_{m \in \NN}$ be the sequence of orthonormal polynomials with respect to the weight $w$. In the classical quadrature theory,  a possible choice of integration nodes is to take the zeros of the polynomials $p_m(w)$.
Denote by $x_{m,k}$, $1 \le k \le m$, the increasing sequence of zeros of 	$p_m(w)$. These zeros are contained in the interval $(0,4m)$ and $x_m \approx 4m - m^{1/3}$ 
(see, e.g., \cite{MM2003,MO2001}).
These zeros are located as 
\begin{equation}\label{zeros-location}
	\frac{C'}{m}  < x_{m,1} < \cdots <  x_{m,m} \le 4m - Cm^{1/3}, 
\end{equation}
with positive constants $C,C'$ independent of $m$.  It is useful to remark that
\begin{equation} \label{x_k+1 - x_k}
	x_{m,k} \, \asymp  \,  \frac{k^2}{m}, \ \ 	
	d_{m,k} \, \asymp  \,  \brac{\frac{x_{m,k}}{4m - x_{m,k}}}^{1/2}, \ k=1,...,m, 
\end{equation} 
where $d_{m,k}:= x_{m,k} - x_{m,k-1}$ (with the convention $x_0:=0$) is the distance between consecutive zeros of  the polynomial $p_m(w)$.

 For a continuous function on $\RRp$, the Gauss-Laguerre quadrature is defined as
\begin{equation} \label{G-Q_mf}
	Q^{\rm L}_mf: = \sum_{k=1}^m \lambda_{m,k}(w) f(x_{m,k}), 
\end{equation} 
where $\lambda_{m,k}(w)$ are the corresponding Cotes numbers, defined as
 \begin{equation} \label{}
 	\lambda_{m,k}:= \int_{\RRp} \ell_k(x) w(x) \, \rd x, \ \ \
 	\ell_k(x):= \prod_{1\le i \le m, \, i \not= k  } \frac{x - x_j}{x_k - x_i}.
 \end{equation}

Unfortunately, this quadrature  gives a very low convergence rate for functions from $\bW^r_{1,w_r}(\RRp)$, see Remark \ref{Comment on G-quadrature} below. 
In \cite{MM2001}, the authors proposed truncated Gauss-Laguerre quadratures which crucially  improve the convergence rate of  $\Int_n\big(\bW^r_{1,w_r}(\RRp)\big)$ as shown in Theorem \ref{thm:Q_n-d=1} below. 

Throughout this paper, we fix a number $\theta$ with $0 < \theta < 1$, and denote by $j(m)$ the largest integer satisfying $x_{m,j(m)} \le 4\theta m$.
By  \eqref{zeros-location} and \eqref{x_k+1 - x_k}, for $m$ sufficiently large we have that 
\begin{equation} \label{<j(m)<}
Cm \le j(m) \le m
\end{equation} 
with a positive constant $C$ depending on $\alpha, \theta$ only. 

For a continuous function on $\RRp$, consider 
the truncated Gauss-Laguerre quadrature 
\begin{equation} \label{G-Q_mf-TG}
	Q^{\rm{TL}}_{ j(m)} f: = \sum_{k=1}^{ j(m)} \lambda_{m,k}(w) f(x_{m,k}). 
\end{equation} 
Notice that the number $j(m)$ of samples in the quadrature $Q^{\rm{TL}}_{ j(m)}f$ is strictly smaller than $m$ -- the number of samples in the quadrature  $Q^{\rm L}_mf$.   However, due to \eqref{<j(m)<} it has the asymptotic order as $ j(m)\asymp m$ when $m$ going to infinity.

\begin{theorem} \label{thm:Q_n-d=1}
	For any $n \in \NN$, let $m_n$ be the largest integer such that $ j(m_n) \le n$. Then we have that
	\begin{equation}\label{Q_n-d=1}
		n^{-3r/4}
		\ll
		\Int_n\big(\bW^r_{1,w_r}(\RRp)\big) 
		\le
		\sup_{f\in \bW^r_{1,w_r}(\RRp)} \bigg|\int_{\RRp}f(x) w(x) \rd x - Q^{\rm{TL}}_{ j(m_n)}f\bigg| 
		\ll
		n^{- r/2}.
	\end{equation}
\end{theorem}

\begin{proof} 
For $f \in W^r_{1,w_r}(\RRp)$, there holds the inequality
	\begin{equation}\label{MM2001}
	\bigg|\int_{\RRp}f(x) w(x) \rd x - Q^{\rm{TL}}_{j(m)}f\bigg| 
	\le 
C\brac{m^{-r/2} \norm{f^{(r)}}{L_{1, w_r}(\RRp)} + e^{-Km}\norm{f}{L_{1,w_r}(\RRp)} } 
\end{equation}
with some constants $C$ and $K$ independent of $m$ and $f$, see \cite[Corollary 2.4]{MM2001}.  The inequality \eqref{MM2001} implies the upper bound in \eqref{Q_n-d=1}.

The lower bound in \eqref{Q_n-d=1} is  a particular case  in Theorem \ref{theorem:Int_n(W)} below. Since its proof is much simpler for the case $d=1$, let us process it separately.  Let $\brab{\xi_1,...,\xi_n} \subset \RRp$ be arbitrary $n$ points. For a given $n \in \NN$, we put $\delta = n^{-1/2}$ and $t_j = \delta j$, $j \in \NN_0$. Then there is $i \in \NN$ with 
$n + 1 \le  i \le 2n + 2$ such that the interval $(t_{i-1}, t_i)$ does not contain any point from the set $\brab{\xi_1,...,\xi_n}$.
Take a nonnegative function $\phi \in C^\infty_0([0,1])$, $\phi \not= 0$, and put
	\begin{equation*}\label{b_s}
b_0:= \int_0^1 \phi(y) \rd y > 0, \quad b_s :=  \int_0^1 |\phi^{(s)}(y)| \rd y, \ s = 1,...,r.
\end{equation*}		
Define the functions $g$ and $h$ on $\RRp$ by
\begin{equation*}\label{g}
g(x):= 
\begin{cases}
\phi(\delta^{-1}(x - t_{i-1})), & \ \ x \in (t_{i-1}, t_i), \\
0, & \ \ \text{otherwise},	
\end{cases}	
\end{equation*}
and
	\begin{equation}\label{h,v}
		h(x):= (gw^{-1})(x).		
	\end{equation}
Let us estimate the norm $\norm{h}{W^r_{1,w_r}(\RRp)}$. For a given $k \in \NN_0$ with $0 \le k \le r$,  we have
	\begin{equation}\label{h^(s)}
	h^{(k)} = (gw^{-1})^{(k)} = \sum_{s=0}^k \binom{k}{s} g^{(k-s)}(w^{-1})^{(s)}. 		
\end{equation}
By a direct computation we find that   for $x \in \RRp$,
	\begin{equation}\label{w^(s)}
(w^{-1})^{(s)}(x) = (w^{-1})(x) x^{-s}\prod_{j=0}^{s-1}(-\alpha -j).
\end{equation}
Hence,  we obtain
	\begin{equation}\label{h^(s)w(x)}
	h^{(k)}(x) w_r(x)= \sum_{s=0}^k \binom{k}{s} g^{(k-s)}(x) x^{r/2-s}
	\prod_{j=0}^{s-1}(-\alpha -j) 
\end{equation}
which implies that 
	\begin{equation}\nonumber
	\int_{\RRp}|h^{(k)} w_r|(x) \rd x 
	\le  C \max_{0\le s \le k} \ \int_{t_{i-1}}^{t_i} x^{r/2-s}|g^{(k-s)}(x)| \rd x.		
\end{equation}
From  the inequalities $n^{1/2} \le x \le (2n + 2)n^{-1/2}$ for $x \in [t_{i-1},t_i]$, and 
\begin{equation*}\label{int-g^(k-s)}
  \int_{t_{i-1}}^{t_i} |g^{(k-s)}(x)| \rd x = b_{k-s} \delta^{-k+s+1} = 	b_{k-s} n^{(k-s-1)/2},
\end{equation*}
we derive
	\begin{equation*}\label{int-h^(s)w_r}
		\begin{aligned}
	\int_{\RRp}|h^{(k)} w_r|(x) \rd x 
	&\le  C \max_{0\le s \le k}   \int_{t_{i-1}}^{t_i} x^{r/2-s}|g^{(k-s)}(x)| \rd x
	\\
	& \le  C \max_{0\le s \le k}   \brac{n^{1/2}}^{r/2-s}
	\int_{t_{i-1}}^{t_i}|g^{(k-s)}(x)| \rd x
	\\& 
	\le  C n^{r/4}\max_{0\le s \le k} n^{-s/2}		 n^{(k-s-1)/2}
	\\
	&= C n^{r/4} n^{(k-1)/2} \le  C n^{3r/4 - 1/2}.		
	\end{aligned}
\end{equation*}
If we define 
	\begin{equation*}\label{h-bar}
	\bar{h}:=  C^{-1} n^{-(3r/4 - 1/2)}h,
\end{equation*}
then $\bar{h}$ is nonnegative, 
$\bar{h} \in \bW^r_{1,w_r}(\RRp)$, $\sup \bar{h} \subset (t_{i-1},t_i)$ and
	\begin{equation*}\label{int-h^(s)w2}
	\begin{aligned}	
		\int_{\RRp}(\bar{h}w)(x) \rd x 
	&=  C^{-1} n^{-(3r/4 - 1/2)}\int_{t_{i-1}}^{t_i}g(x) \rd x
	\\
	&
	=  C^{-1} n^{-(3r/4 - 1/2)} b_0 \delta  \gg n^{-3r/4}
	\end{aligned}
\end{equation*}
 Since the interval $(t_{i-1}, t_i)$ does not contain any point from the set $\brab{\xi_1,...,\xi_n}$, we have $\bar{h}(\xi_k) = 0$, $k = 1,...,n$. Hence, by the inequality \eqref{Int_n>},
	\begin{equation}\nonumber
		\Int_n\big(\bW^r_{1,w_r}(\RRp)\big) 	\ge  \int_{\RRp}\bar{h}(x) v(x) \rd x
	 \gg  n^{- 3r/4}.
\end{equation}
The lower bound in \eqref{Q_n-d=1} is proven.
	\hfill
\end{proof}

\begin{remark} \label{Comment on G-quadrature}
{\rm	
For the full Gaussian quadratures $Q^{\rm L}_n$,  it has been proven in \cite[Theorem 2.8]{MM2003}	the convergence rate
\begin{equation}\nonumber
	\sup_{f\in \bW^1_{1,w_r}(\RRp)}	\bigg|\int_{\RRp}f(x) w(x) \rd x  - Q^{\rm L}_n f\bigg|
	\asymp
	n^{-1/6}
\end{equation}
which is much worse than the convergence rate of 
$	\Int_n\big(\bW^1_{1,w_r}(\RRp)\big) \ll 	n^{-1/2}$ as 
in \eqref{Q_n-d=1} for $r = 1$.
}
\end{remark}

Let $f \in W^r_{1,w_r}(\RR)$. We define  $f_{+1}(x):=f(x)$ and
$f_{-1}(x):=f(-x)$ for $x \in \RRp$. 
Then the functions $f_{+1}$  
and
$f_{-1}$ belong to 
$W^r_{1,w_r}(\RRp)$ satisfying
\begin{equation}\label{norm<}
\norm{f_{\pm 1}}{W^r_{1,w_r}(\RRp)} \le \norm{f}{W^r_{1,w_r}(\RR)},
\end{equation} 
and 
\begin{equation}\label{int=}
\int_{\RR}f(x) w(x) \rd x  = \int_{\RRp}f_{+1}(x) w(x) \rd x  + \int_{\RRp}f_{-1}(x) w(x) \rd x. 
\end{equation}

For a continuous function on $\RR$, consider 
the truncated Gauss-Laplace quadrature 
\begin{equation} \label{G-Q_mf-2TL}
	Q^{\rm{2TL}}_{ j(m)} f
	: = 
	\sum_{k=1}^{ j(m)} \lambda_{m,k}(w) \brac{f(x_{m,k}) + f(-x_{m,k})}. 
\end{equation} 

From Theorem~\ref{thm:Q_n-d=1} and \eqref{norm<}, \eqref{int=} we derive
\begin{theorem} \label{thm:Q_n-d=1-Laplace}
	For any $n \in \NN$, let $m_n$ be the largest integer such that $2 j(m_n) \le n$. Then the quadratures	$Q^{\rm{2TL}}_{ j(m_n)} \in \Qq_n$, $n \in \NN$, are  asymptotically optimal  for $\bW^r_{1,w_r}(\RR)$ and
	\begin{equation}\label{Q_n-d=1-Laplace}
			n^{-3r/4}
		\ll
		\Int_n\big(\bW^r_{1,w_r}(\RR)\big) 
		\le
		\sup_{f\in \bW^r_{1,w_r}(\RR)} \bigg|\int_{\RR}f(x) w(x) \rd x - Q^{\rm{2TL}}_{ j(m_n)}f\bigg| 
		\ll
		n^{- r/2}.
	\end{equation}
\end{theorem}


	\section{Weighted integration in high dimension}
\label{Weighted integration in high dimension}

In this section, for high-dimensional numerical integration ($d \ge 2$), we prove upper and lower bounds of   $\Int_n\big(\bW^r_{1,\bw_r}(\UUd)\big)$ and construct  quadratures based on  step-hyperbolic-corner  or step-hyperbolic-cross grids of integration nodes which give the upper bounds.

	Assume that there exists a sequence of quadratures $\brac{Q_{2^k}}_{k \in \NN_0}$ with
\begin{equation}\label{Q_2^kf}
	Q_{2^k}f: = \sum_{s=1}^{2^k} \lambda_{k,s} f(x_{k,s}), \ \ \{x_{k,1},\ldots,x_{k,2^k}\}\subset \UU,
\end{equation}
such that 
\begin{equation}\label{IntErrorR}
	\bigg|\int_{\UU} f(x) w(x)\rd x - Q_{2^k}f\bigg| 
	\leq C 2^{- ak} \|f\|_{W^r_{1,w_r}(\UU)}, 
	\ \  \ k \in \NN_0,  \ \ f \in W^r_{1,w_r}(\UU), 
\end{equation}
for some  number $a>0$ and constant $C>0$.

Based on a sequence $\brac{Q_{2^k}}_{k \in \NN_0}$ of the form \eqref{Q_2^kf} satisfying \eqref{IntErrorR}, we construct quadratures on $\UUd$ by using the well-known Smolyak algorithm. We define for $k \in \NN_0$, the one-dimensional operators
\begin{equation*}\label{DeltaI}
\Delta_k^Q:= Q_{2^k} - Q_{2^{k-1}}, \  k >0, \ \ \Delta_0^Q:= Q_1.  
\end{equation*}
For $\bk \in \NNd$, the $d$-dimensional operator  $\Delta_\bk^Q$  is defined as the tensor  product of one-dimensional operators:
\begin{equation}\label{tensor-product}
	Q_{2^\bk}:= \bigotimes_{i=1}^d Q_{2^{k_i}} , \ \	
	\Delta_\bk^Q:= \bigotimes_{i=1}^d \Delta_{k_i}^Q, \ \ 	
	E_\bk^Q:= \bigotimes_{i=1}^d E_{k_i}^Q, 
\end{equation}
where $2^\bk:= (2^{k_1},\cdots, 2^{k_d})$ and 
the univariate operators  $\Delta_{k_j}^Q$ 
 are successively applied to the univariate functions $\bigotimes_{i<j} \Delta_{k_i}^Q(f)$ by considering them  as 
functions of  variable $x_j$ with the other variables held fixed. The operators$\Delta_\bk^Q$  are well-defined for continuous functions on $\UUd$, in particular for ones from $W^r_{1,\bw_r}(\UUd)$.

Notice that if $f$ is a continuous function on $\UUd$, then we have
\begin{equation}\label{Delta_bk}
\Delta_\bk^Q f =  \sum_{e \subset \brab{1,...,d}} (-1)^{d - |e|}Q_{2^{\bk(e)}} f
	= \sum_{e \subset \brab{1,...,d}} (-1)^{d - |e|}\sum_{\bs=\bone}^{2^{\bk(e)}} \lambda_{\bk(e),\bs} f(\bx_{\bk(e),\bs}), 
\end{equation}
where  $\bk(e) \in \NNd_0$ is defined by $k(e)_i = k_i$, $i \in e$, and 	$k(e)_i = \max(k_i-1,0)$, $i \not\in e$,  
$$
\bx_{\bk,\bs}:= \brac{x_{k_1,s_1},...,x_{k_d,s_d}}, \ \ \ 
\lambda_{\bk,\bs}:= \prod_{i=1}^d \lambda_{k_i,s_i} ,
$$
and the summation $\sum_{\bs=\bone}^{2^\bk}$ means that the sum is taken over all $\bs$ such that $\bone \le \bs \le 2^\bk$.

We now define an  algorithm for quadrature on sparse grids adopted from the alogorithm for sampling recovery initiated by Smolyak (for detail see \cite[Sections 4.2 and 5.3]{DTU18B}). For $\xi > 0$, we define the operator
	\begin{equation}\nonumber
	Q_\xi
	:= 	\sum_{|\bk|_1 \le \xi } \Delta_{\bk}^Q.
\end{equation}	
From \eqref{Delta_bk} we can see that $Q_\xi$ is a quadrature on $\UUd$ of the form  \eqref{Q_nf-introduction}:
\begin{equation}\label{Q_xi}
	Q_\xi f
	= 	\sum_{|\bk|_1 \le \xi } \ \sum_{e \subset \brab{1,...,d}} (-1)^{d - |e|}\ \sum_{\bs=\bone}^{2^{\bk(e)}} \lambda_{\bk(e),\bs} f(\bx_{\bk(e),\bs})
	= \sum_{(\bk,e,\bs) \in G(\xi)} \lambda_{\bk,e,\bs} f(\bx_{\bk,e,\bs}), 
\end{equation}
where 
$$
\bx_{\bk,e,\bs}:= \bx_{\bk(e),\bs}, \quad \lambda_{\bk,e,\bs}:= (-1)^{d - |e|}\lambda_{\bk(e),\bs}
$$ 
and 
\begin{equation}\nonumber
G(\xi)	:= \brab{(\bk,e,\bs): \ |\bk|_1 \le \xi, \,   e \subset \brab{1,...,d}, \,  \bone \le \bs \le \bk(e)}
\end{equation}
is a finite set.
The set of integration nodes in this quadrature 
\begin{equation}\nonumber
H(\xi):=\brab{\bx_{\bk,e,\bs}}_{(\bk,e,\bs) \in G(\xi)}
\end{equation}
is a step hyperbolic cross   in the function domain $\UUd$. 
The number of integration nodes in the quadrature $Q_\xi$ is 
\begin{equation}\nonumber
	|G(\xi)|
	= 	\sum_{|\bk|_1 \le \xi } \ \sum_{e \subset \brab{1,...,d}}2^{|\bk(e)|_1}
\end{equation}
which can be estimated as 
\begin{equation}\label{|G(xi)|}
	|G(\xi)|
	\asymp	\sum_{|\bk|_1 \le \xi }2^{|\bk|_1} \ \asymp \ 2^\xi \xi^{d - 1}, \ \ \xi \ge 1.
\end{equation}
This quadrature plays a basic role in the proof of the upper bound in the main results of the present paper \eqref{Int_n(W^r_1)}.

In a way analogous to the proof of  \cite[Lemma 3.5]{DD2023}, we can prove 

\begin{lemma} \label{lemma:Q_xi-error}
	Under the assumption \eqref{Q_2^kf}--\eqref{IntErrorR}, we have that
\begin{equation}\label{upperbound}
	\bigg|\int_{\UUd} f(\bx) \bw(\bx)\rd \bx  -  Q_\xi f\bigg| 
	\leq C 2^{- a\xi} \xi^{d - 1} \|f\|_{W^r_{1,\bw_r}(\UUd)}, 
	\ \ \xi \ge 1, \ \ f \in W^r_{1,\bw_r}(\UUd). 
\end{equation}
\end{lemma}

\begin{remark} \label{remark1}
	{\rm	
		Importantly, from Theorems \ref{thm:Q_n-d=1} and \ref{thm:Q_n-d=1-Laplace} it follows that the truncated Gaussian-Laguerre quadratures $Q^{\rm{TL}}_{ j(m)}$ and $Q^{\rm{2TL}}_{ j(m)}$ 
		form  a sequence $\brac{Q_{2^k}}_{k \in \NN_0}$ of the form \eqref{Q_2^kf} satisfying  \eqref{IntErrorR} with $a = r/2$. 
	}
\end{remark}

\begin{theorem} \label{theorem:Int_n(W)}
	We have that
	\begin{equation}\label{Int_n(W)}
n^{-3r/4} (\log n)^{3r(d-1)/4}
	\ll	
	\Int_n(\bW^r_{1,\bw_r}(\UUd)) 
	\ll 
n^{-r/2} (\log n)^{(r/2 +1)(d-1)}.
	\end{equation}
\end{theorem}

\begin{proof}  
	We prove the theorem for the case $\UUd=\RRdp$. The case $\UUd=\RRd$ can be established in an analogous manner with some modifications.
Let us first prove the upper bound in \eqref{Int_n(W)}. We  will construct a quadrature of the form \eqref{Q_xi} which realizes it.	
In order to do this, we take the truncated Gaussian quadrature 	$Q^{\rm{TL}}_{j(m)}f$ defined in \eqref{G-Q_mf}. For every $k \in \NN_0$, let $m_k$ be the largest number such that $j(m_k) \le 2^k$.  Then we have 
$j(m_k) \asymp 2^k$. 
For the sequence of quadratures  $\brac{Q_{2^k}}_{k \in \NN_0}$ with 
$$
Q_{2^k}:= Q^{\rm{TL}}_{ j(m_k)} \in \Qq_{2^k},
$$
 from Theorem \ref{thm:Q_n-d=1} it follows that
\begin{equation*}\label{IntErrorR2}
	\bigg|\int_{\RR} f(x) w(x)\rd x - Q_{2^k}f\bigg| 
	\leq C 2^{- \frac{1}{2}r k} \|f\|_{W^r_{1,w_r}(\RR)}, 
	\ \  \ k \in \NN_0,  \ \ f \in W^r_{1,w_r}(\RR). 
\end{equation*}
This means that the assumption \eqref{Q_2^kf}--\eqref{IntErrorR} holds for $a = r_\lambda$. To prove the upper bound in \eqref{Int_n(W)} we approximate the integral 
$$
\int_{\RRdp} f(\bx) \bw(\bx)\rd \bx
$$
  by the quadrature $Q_{\xi}$ which is formed from the sequence $\brac{Q_{2^k}}_{k \in \NN_0}$.
For every $n \in \NN$, let $\xi_n$ be the largest number such that $|G(\xi_n)| \le n$.  Then the corresponding operator $Q_{\xi_n}$ defines a quadrature belonging to $\Qq_n$. From \eqref{|G(xi)|} it follows 
$$
\ 2^{\xi_n} \xi_n^{d - 1} \asymp |G(\xi_n)|  \asymp n.
$$ 
Hence we deduce the asymptotic equivalences
$$
\ 2^{-\xi_n}  \asymp n^{-1} (\log n)^{d-1}, \ \  \xi_n \asymp \log n,
$$
which together with Lemma \ref{lemma:Q_xi-error}  yield that
	\begin{equation*}\label{Q_xi-error}
		\begin{aligned}
		\Int_n(\bW^r_{1,\bw_r}(\RRdp)) 
		&\le 
	\sup_{f\in \bW^r_{1,\bw_r}(\RRdp)}	
	\bigg|\int_{\RRdp} f(\bx) \bw_r(\bx)\rd \bx  -  Q_{\xi_n}f\bigg| 
		\\
		&
		\leq 
		C 2^{- \frac{1}{2} r \xi_n} \xi_n^{d - 1}
		\asymp  	n^{-r/2} (\log n)^{(r/2 + 1)(d-1)}.
		\end{aligned}
	\end{equation*}
The upper bound in \eqref{Int_n(W)} is proven.

We now prove the lower bound  in \eqref{Int_n(W)} by using the inequality \eqref{Int_n>} in Lemma \ref{lemma:Int_n>}. For $M \ge 1$, we define the set 
	\begin{equation*}\label{Gamma-set}
\Gamma_d(M):= \brab{\bs \in \NNd: \, \prod_{i=1}^d s_i \le 2M, \ s_i \ge M^{1/d}, \ i=1,...,d}.
\end{equation*}
Then we have 
	\begin{equation}\label{|Gamma|}
	|\Gamma_d(M)| \asymp M (\log M)^{d-1}, \ \ M>1.
\end{equation}

For a given $n \in \NN$, let $\brab{\bxi_1,...,\bxi_n} \subset \RRdp$ be arbitrary  $n$ points. Denote by $M_n$  the smallest number such that $|\Gamma_d(M_n)| \ge n + 1$. We define the $d$-parallelepiped $K_\bs$ for $\bs \in \NNd_0$ of size 
$$
\delta:= M_n^{- \frac{1}{2d}}
$$
 by
	\begin{equation*}\label{K_bs}
K_\bs:= \prod_{i=1}^d K_{s_i}, \ \ K_{s_i}:= (\delta s_i, \delta s_{i-1}).
\end{equation*}
 Since  $|\Gamma_d(M_n)| > n$, there exists a multi-index 
$\bs \in \Gamma_d(M_n)$ such that $K_\bs$ does not contain any point from $\brab{\bxi_1,...,\bxi_n}$.

As in the proof of Theorem \ref{thm:Q_n-d=1}, we take a nonnegative function $\phi \in C^\infty_0([0,1])$, $\phi \not= 0$, and put
\begin{equation}\label{b_s-d}
	b_0:= \int_0^1 \phi(y) \rd y > 0, \quad b_s :=  \int_0^1 |\phi^{(s)}(y)| \rd y, \ s = 1,...,r.
\end{equation}		
For $i= 1,...,d$, we define the univariate functions $g_i$ in variable $x_i$ by
\begin{equation}\label{g_i}
	g_i(x_i):= 
	\begin{cases}
		\varphi(\delta^{-1}(x_i - \delta s_{i-1})), & \ \ x_i  \in K_{s_i}, \\
		0, & \ \ \text{otherwise}.	
	\end{cases}	
\end{equation}
Then the multivariate functions $g$ and $h$ on $\RRdp$ are defined by
	\begin{equation*}\label{g(bx)}
	g(\bx):= \prod_{i=1}^d g_i(x_i),
\end{equation*}
and  
\begin{equation}\label{h(bx)}
	h(\bx):= (g \bw^{-1})(\bx)= \prod_{i=1}^d g_i(x_i)w^{-1}(x_i)=:
	\prod_{i=1}^d h_i(x_i).
\end{equation}
Let us estimate the norm $\norm{h}{W^r_{1,\bw_r}(\RRdp)}$. For every $\bk \in \NNd_0$ with 
$0 \le |\bk|_\infty \le r$,  we prove the inequality
\begin{equation}\label{int-D^br}
	\int_{\RRdp}\big|(D^{\bk} h) \bw_r\big|(\bx) \rd \bx 
	\le  C   M_n^{(r - 1)/2}.
\end{equation}
We have
\begin{equation}\label{D^bk h}
	D^\bk h  = \prod_{i=1}^d h_i^{(k_i)}. 		
\end{equation}
Similarly to \eqref{h^(s)}--\eqref{h^(s)w(x)} we derive that for every $i = 1,...,d$,
\begin{equation*}\label{h^(s-i)w(x_i)}
	h_i^{(k_i)}(x_i) w_r(x_i)
	= \sum_{\nu_i=0}^{k_i} \binom{k_i}{\nu_i} g_i^{(k_i- \nu_i)}(x_i) (w^{-1})(x_i) x_i^{r/2-\nu_i}
		\prod_{\eta_i=0}^{\nu_i-1}(-\alpha -j).		
\end{equation*}
This together with \eqref{b_s-d}--\eqref{g_i} and the inequalities $s_i \ge M_n^{\frac{1}{d}}$  yields that
\begin{equation}\label{int-h^(s)w-d>1}
	\begin{aligned}
			\int_{\UU}\big|h_i^{(k_i)}(x_i) w_r(x_i)\big|\rd x_i
	&\le  C \max_{0\le \nu_i \le k_i} \ \max_{1 \le \eta_i \le \nu_i}  
	\int_{K_{s_i}}|x_i|^{r/2-\nu_i}\big|g^{(k_i-\nu_i)}(x_i)\big| \rd x_i 
	\\
	&\le  C \max_{0\le \nu_i \le k_i} (\delta s_i)^{r/2-\nu_i}
	\int_{K_{s_i}}\big|g^{(k_i-\nu_i)}(x_i)\big| \rd x_i 
	\\
	&\le C \max_{0\le \nu_i \le k_i} (\delta s_i)^{r/2-\nu_i}
	\delta^{- k_i + \nu_i +1} b_{k_i - \nu_i}
	\\
	&= C \delta^{- k_i +1} (\delta s_i)^{r/2}\max_{0\le \nu_i \le k_i}(\delta s_i)^{-\nu_i}\delta^{-\nu_i}. 
\end{aligned}
\end{equation}
Since $s_i \ge M_n^{\frac{1}{d}}$ and $\delta:= M_n^{-\frac{1}{2d}}$, we have that
$(\delta s_i)^{-\nu_i}\delta^{-\nu_i} \le 1$, and consequently, 
 the estimates \eqref{int-h^(s)w-d>1} and the inequalities  $0 \le k_i \le r$ and 
$\delta s_i \ge 1$ yield that
\begin{equation}\nonumber
		\int_{\UU}\big|h_i^{(k_i)}(x_i) w_r(x_i)\big|\rd x_i
		\le  C M_n^{\frac{r-1}{2d}}(\delta s_i)^{r/2}. 
\end{equation}
Hence, by \eqref{D^bk h} and the inequality $\prod_{i=1}^d s_i \le 2M$ we deduce 
\begin{equation}\nonumber
		\int_{\RRdp}|(D^\bk h) \bw_r|(\bx) \rd \bx 
		= \prod_{i=1}^d  \int_{\UU}\big|h^{(k_i)}(x_i)w_r(x_i)\big|\rd x_i
		 \le  C  M_n^{3r/4 - 1/2}.
\end{equation}
which completes the proof of the inequality \eqref{int-D^br}.
This  inequality means that $h \in W^r_{1,\bw_r}(\RRdp)$ and  
$$
\norm{h}{W^r_{1,\bw_r}(\RRdp)} \le  C   M_n^{3r/4 - 1/2}.
$$
 If  we define 
\begin{equation*}\label{h-bar}
	\bar{h}:=  C^{-1} M_n^{-({3r/4 - 1/2})}h,
\end{equation*}
then $\bar{h}$ is nonnegative, $\bar{h} \in \bW^r_{1,w_r}(\RR)$, $\sup \bar{h} \subset K_\bs$ and by \eqref{b_s-d}--\eqref{h(bx)},
\begin{equation}\nonumber
	\begin{aligned}	
		\int_{\RRdp}(\bar{h}w)(\bx) \rd \bx 
		&=  C^{-1} M_n^{-({3r/4 - 1/2})}\int_{\RRdp}(h w)(\bx) \rd \bx = \prod_{i=1}^d \int_{K_{s_i}} g_i(x_i) \rd x_i
		\\
		&
		=  C^{-1} M_n^{-({3r/4 - 1/2})}\brac{b_0 \delta}^d =  C' M_n^{- 3r/4}.
	\end{aligned}
\end{equation}
From the definition of $M_n$ and \eqref{|Gamma|} it follows that
$$M_n(\log M_n)^{d-1} \asymp |\Gamma(M_n)| \asymp n,$$ 
 which implies that $M_n^{-1} \asymp n^{-1} (\log n)^{d-1}$. This allows to receive the estimate
\begin{equation}\label{int-h-bar2}
		\int_{\RRdp}(\bar{h}w)(\bx) \rd \bx 
		=  C' M_n^{-3r/4}
		\gg
		n^{-3r/4} (\log n)^{3r(d-1)/4}.
\end{equation}
 Since the interval $K_\bs$ does not contain any point from the set $\brab{\bxi_1,...,\bxi_n}$ which has been arbitrarily chosen, we have 
 $$\bar{h}(\bxi_k) = 0, \ \ k = 1,...,n.
 $$
  Hence, by  the inequality \eqref{Int_n>} and \eqref{int-h-bar2} we have that
\begin{equation*}\label{int-h-bar3}
	\Int_n(\bW^r_{1,\bw_r}(\RRdp)) 
	\ge
	 \int_{\RRdp}\bar{h}(\bx) \bw_r(\bx) \rd \bx
	\gg  n^{-3r/4} (\log n)^{3r(d-1)/4}.
\end{equation*}
The lower bound in \eqref{Int_n(W)} is proven.
	\hfill
\end{proof}

\begin{remark}
	{\rm	
	For numerical integration of functions  in $W^r_{1,\bw_r}(\RRdp)$, the set of integration nodes $H(\xi)$ in the quadratures $Q_\xi$ which are formed from the truncated non-equidistant zeros of the Laguerre polynomials $p_m(w)$,
		is a  step hyperbolic corner  on the function domain $\RRdp$. Similarly, for numerical integration of functions  in $W^r_{1,\bw_r}(\RRd)$, the set of integration nodes $H(\xi)$ in the quadratures $Q_\xi$ which are formed from the truncated non-equidistant zeros of the Laguerre polynomials $p_m(w)$ and their symmetrically transformed points with respect to the origin, is a  step hyperbolic cross   on the function domain $\RRd$.
		This is a contrast to the  classical theory of  approximation of multivariate periodic functions having mixed smoothness for which the classical step hyperbolic crosses of integer points are on the frequency domain $\ZZd$ (see, e.g., \cite[Section 2.3]{DTU18B} for detail).  
		The set  $H(\xi)$ also completely differs from the classical Smolyak grids of fractional dyadic points on the function domain $[-1,1]^d$ 
		which are used in numerical integration for functions having a mixed smoothness (see, e.g., \cite[Section~5.3]{DTU18B} for detail). 		
}
\end{remark}

\medskip
\noindent
{\bf Acknowledgments:} 
This work is funded by the Vietnam National Foundation for Science and Technology Development (NAFOSTED) in the frame of the NAFOSTED-SNSF Joint Research Project under  
Grant IZVSZ2$_{ - }$229568.  
A part of this work was done when  the author was working at the Vietnam Institute for Advanced Study in Mathematics (VIASM). He would like to thank  the VIASM  for providing a fruitful research environment and working condition. 

\bibliographystyle{abbrv}
\bibliography{W-La-Integration}
\end{document}